%% file: PosFactor.tex
\documentclass[twoside,11pt]{article}

\usepackage{times}
\usepackage{graphicx}
\usepackage{subfigure} 

\usepackage{amsmath}%
\usepackage{amsfonts}%
\usepackage{amssymb}%
\usepackage{mathrsfs}

\usepackage{algorithm}
\usepackage{algorithmic}

\usepackage{hyperref}
\usepackage{natbib}

\usepackage{jmlr2e}

\newtheorem{assumption}{Assumption}

\def\Image{\textnormal{Im}}
\def\Nplus{\mathbb{N}_+}
\def\card{\textnormal{card}}
\DeclareMathOperator*{\argmin}{arg\,min}
\def\Re{\mathbb{R}}

\ShortHeadings{Globally Optimal Positively Homogeneous Factorizations}{Haeffele and Vidal}
\firstpageno{1}

\begin{document}

\title{Global Optimality in Tensor Factorization, Deep Learning, and Beyond}

\author{\name Benjamin D. Haeffele \email bhaeffele@jhu.edu\\ 
\name Ren\'e Vidal \email rvidal@jhu.edu\\ 
\addr Department of Biomedial Engineering\\
Johns Hopkins University\\
Baltimore, MD 21218, USA 
}

\editor{}

\maketitle

\sloppy

\input{PosFactor_abstract.tex}

\input{PosFactor_intro2.tex}

\input{PosFactor_preliminaries.tex}

\input{PosFactor_problem.tex}

\input{PosFactor_mainresults.tex}

\input{PosFactor_conclusions.tex}

\bibliography{posfactor,../biblio/sparse,../biblio/vidal}

\end{document}

%% file: PosFactor_abstract.tex
\begin{abstract}
Techniques involving factorization are found in a wide range of applications and have enjoyed significant empirical success in many fields.  However, common to a vast majority of these problems is the significant disadvantage that the associated optimization problems are typically non-convex due to a multilinear form or other convexity destroying transformation.  Here we build on ideas from convex relaxations of matrix factorizations and present a very general framework which allows for the analysis of a wide range of non-convex factorization problems - including matrix factorization, tensor factorization, and deep neural network training formulations.  We derive sufficient conditions to guarantee that a local minimum of the non-convex optimization problem is a global minimum and show that if the size of the factorized variables is large enough then from any initialization it is possible to find a global minimizer using a purely local descent algorithm. Our framework also provides a partial theoretical justification for the increasingly common use of Rectified Linear Units (ReLUs) in deep neural networks and offers guidance on deep network architectures and regularization strategies to facilitate efficient optimization.


\end{abstract} 

%% file: PosFactor_intro2.tex
\section{Introduction}


Models involving factorization or decomposition are ubiquitous across a wide variety of technical fields and application areas.  As a simple example relevant to machine learning, various forms of \textit{\textbf{matrix factorization}} are used in classical dimensionality reduction techniques such as Principle Component Analysis (PCA) and in more recent methods like non-negative matrix factorization or dictionary learning \citep{Lee:Nature1999,Elad:TSP06,Mairal:JMLR2010}.  In a typical matrix factorization problem, we might seek to find matrices $(U,V)$ such that the product $UV^T$ closely approximates a given data matrix $Y$ while at the same time requiring that $U$ and $V$ satisfy certain properties (e.g., non-negativity, sparseness, etc.).  This naturally leads to an optimization problem of the form
\begin{equation}
\label{eq:intro_matrix_fact_obj}
\min_{U,V} \ell(Y,UV^T) + \Theta(U,V)
\end{equation}
where $\ell$ is some function that measures how closely $Y$ is approximated by $UV^T$ and $\Theta$ is a regularization function to enforce the desired properties in $U$ and $V$. Unfortunately, aside from a few special cases (e.g., PCA), a vast majority of matrix factorization models suffer from the significant disadvantage that the associated optimization problems are non-convex and very challenging to solve.  For example, in \eqref{eq:intro_matrix_fact_obj} even if we choose $\Theta(U,V)$ to be jointly convex in $(U,V)$ and $\ell(Y,X)$ to be a convex function in $X$, the optimization problem is still typically a non-convex problem in $(U,V)$ due to the composition with the bilinear form $X=UV^T$.

Given this challenge, a common approach is to relax the non-convex factorization problem into a problem which is convex on the product of the factorized matrices, $X=UV^T$.  As a concrete example, in low-rank matrix factorization, one might be interested in solving a problem of the form
\begin{equation}
\label{eq:low_rank_fact}
\min_{U,V} \ell(Y,UV^T) \ \textnormal{subject to rank}(UV^T) \leq r
\end{equation}
where the rank constraint can be easily enforced by limiting the number of columns in the $U$ and $V$ matrices to be less than or equal to $r$.  However, aside from a few special choices of $\ell$, solving \eqref{eq:low_rank_fact} is a NP-hard problem in general.  Instead, one can relax \eqref{eq:low_rank_fact} into a fully convex problem by using a convex regularization that promotes low-rank solutions, such as the nuclear norm $\|X\|_*$, and then solve
\begin{equation}
\label{eq:priorwork_nuclear_norm}
\min_X \ell(Y,X) + \lambda \|X\|_*
\end{equation}
which can be done efficiently if $\ell(Y,X)$ is convex with respect to $X$ \citep{Cai:SJO08,Recht:SIAM10}.  Given a solution to \eqref{eq:priorwork_nuclear_norm}, $X_{opt}$, it is then simple to find a low-rank factorization $UV^T = X_{opt}$ via a singular value decomposition.  Unforunately, however, while the nuclear norm provides a nice convex relaxation for low-rank matrix factorization problems, nuclear norm relaxation does not capture the full generality of problems such as \eqref{eq:intro_matrix_fact_obj} as it does not necessarily ensure that $X_{opt}$ can be 'efficiently' factorized as $X_{opt}=UV^T$ for some $(U,V)$ pair which has the desired properties encouraged by $\Theta(U,V)$ (sparseness, non-negativity, etc.), nor does it provide a means to find the desired factors.  To address these issues, in this paper we consider the task of solving non-convex optimization problems directly in the factorized space and use ideas inspired from the convex relaxation of matrix factorizations as a means to analyze the non-convex factorization problem.  Our framework includes problems such as \eqref{eq:intro_matrix_fact_obj} as a special case but also applies much more broadly to a wide range of non-convex optimization problems; several of which we describe below.

\subsection{Generalized Factorization}

More generally, \textit{\textbf{tensor factorization}} models provide a natural extension to matrix factorization and have been employed in a wide variety of applications \citep{Cichocki:2009,Kolda:SIAMrev2009}.  The resulting optimization problem is similar to matrix factorization, with the difference that we now consider more general factorizations which decompose a multidimensional tensor $Y$ into a set of $K$ different factors $(X^1,\ldots,X^K)$, where each factor is also possibly a multidimensional tensor.  These factors are then combined via an arbitrary multilinear mapping $\Phi(X^1,\ldots,X^K) \approx Y$; i.e., $\Phi$ is a linear function in each $X^i$ term if the other $X^j, \ i\neq j$ terms are held constant.  This model then typically gives optimization problems of the form
\begin{equation}
\label{eq:intro_tensor_fact_obj}
\min_{X^1,\ldots,X^K} \ell(Y,\Phi(X^1,\ldots,X^K)) + \Theta(X^1,\ldots,X^K)
\end{equation}
where again $\ell$ might measure how closely $Y$ is approximated by the tensor $\Phi(X^1,\ldots,X^K)$ and $\Theta$ encourages the factors $(X^1,\ldots,X^K)$ to satisfy certain requirements.  Clearly, \eqref{eq:intro_tensor_fact_obj} is a generalization of \eqref{eq:intro_matrix_fact_obj} by taking $(X^1,X^2) = (U,V)$ and $\Phi(U,V) = UV^T$, and similar to matrix factorization, the optimization problem given by \eqref{eq:intro_tensor_fact_obj} will typically be non-convex regardless of the choice of $\Theta$ and $\ell$ functions due to the multilinear mapping $\Phi$.

While the tensor factorization framework is very general with regards to the dimensionalities of the data and the factors, a tensor factorization usually implies the assumption that the mapping $\Phi$ from the factorized space to the output space (the codomain of $\Phi$) is multilinear.  However, if we consider more general mappings from the factorized space into the output space (i.e., $\Phi$ mappings which are not restricted to be multilinear) then we can capture a much broader array of models in the 'factorized model' family.  For example, in \textit{\textbf{deep neural network training}} the output of the network is typically generated by performing an alternating series of a linear function followed by a non-linear function.  More concretely, if one is given training data consisting of $N$ data points of $d$ dimensional data, $V \in \Re^{N \times d}$, and an associated vector of desired outputs $Y \in \Re^N$, the goal then is to find a set of network parameters $(X^1,\ldots,X^K)$ by solving an optimization problem of the form \eqref{eq:intro_tensor_fact_obj} using a mapping
\begin{equation}
\label{eq:Phi_basic_NN}
\Phi(X^1,\ldots,X^K) = \psi_K(\psi_{K-1} ( \ldots \psi_2(\psi_1(V X^1) X^2) \ldots X^{K-1}) X^K)
\end{equation}
where each $X^i$ factor is an appropriately sized matrix and the $\psi_i(\cdot)$ functions apply some form of non-linearity after each matrix multiplication, e.g., a sigmoid function, rectification, max-pooling.  Note that although here we have shown the linear operations to be simple matrix multiplications for notational simplicity, this is easily generalized to other linear operators (e.g., in a convolutional network each linear operator could be a set of convolutions with a group of various kernels with parameters contained in the $X^i$ variables).

\subsection{Paper Contributions}

Our primary contribution is to extend ideas from convex matrix factorization and present a general framework which allows for a wide variety of factorization problems to be analyzed within a convex formulation.  Specifically, using this convex framework we are able to show that local minima of the non-convex factorization problem achieve the global minimum if they satisfy a simple condition.  Further, we also show that if the factorization is done with factorized variables of sufficient size, then from any initialization it is always possible to reach a global minimizer using purely local descent search strategies. 

Two concepts are key to our analysis framework: 1) the size of the factorized elements is not constrained, but instead fit to the data through regularization (for example, the number of columns in $U$ and $V$ is allowed to change in matrix factorization) 2) we require that the mapping from the factorized elements to the final output, $\Phi$, satisfies a positive homogeneity property.  Interestingly, the deep learning field has increasingly moved to using non-linearities such as Rectified Linear Units (ReLU) and Max-Pooling, both of which satisfy the positive homogeneity property, and it has been noted empirically that both the speed of training the neural network and the overall performance of the network is increased significantly when ReLU non-linearities are used instead of the more traditional hyperbolic tangent or sigmoid non-linearities \citep{Dahl:ICASSP2013,Maas:ICML2013, Krizhevsky:NIPS2012,Zeiler:ICASSP2013}.  We suggest that our framework provides a partial theoretical explanation to this phenomena and also offers directions of future research which might be beneficial in improving the performance of multilayer neural networks.

\section{Prior Work}

Despite the significant empirical success and wide ranging applications of the models discussed above (and many others not discussed), as we have mentioned, a vast majority of the above techniques models suffer from the significant disadvantage that the associated optimization problems are non-convex and very challenging to solve.  As a result, the numerical optimization algorithms often used to solve factorization problems -- including (but certainly not limited to) alternating minimization, gradient descent, stochastic gradient descent, block coordinate descent, back-propagation, and quasi-newton methods -- are typically only guaranteed to converge to a critical point or local minimum of the objective function \citep{Mairal:JMLR2010,Rumelhart:CogModel1988,Ngiam:ICML2011,Wright:1999,Xu:SIAM2013}. The nuclear norm relaxation of low-rank matrix factorization discussed above provides a means to solve factorization problems with reglarization promoting low-rank solutions\footnote{Similar convex relaxation techniques have also been proposed for low-rank tensor factorizations, but in the case of tensors finding a final factorization $X_{opt} = \Phi(X^1,\ldots,X^K)$ from a low-rank tensor can still be a challenging problem \citep{Tomioka:2010,Gandy:InvProb2011}}, but it fails to capture the full generality of problems such as \eqref{eq:intro_matrix_fact_obj} as it does not allow one to find factors, $(U,V)$, with the desired properties encouraged by $\Theta(U,V)$ (sparseness, non-negativity, etc.).  To address this issue, several studies have explored a more general convex relaxation via the matrix norm given by
\begin{equation}
\label{eq:tensor_norm}
\begin{split}
\|X\|_{u,v} &\equiv \inf_{r \in \Nplus} \ \inf_{U,V : UV^T=X} \sum_{i=1}^r \|U_i\|_u \|V_i\|_v \\
	&\equiv \inf_{r \in \Nplus} \ \inf_{U,V : UV^T=X} \sum_{i=1}^r \tfrac{1}{2} (\|U_i\|_u^2+\|V_i\|_v^2) 
\end{split}
\end{equation}
where $(U_i,V_i)$ denotes the $i$'th columns of $U$ and $V$, $\|\cdot\|_u$ and $\|\cdot\|_v$ are arbitrary vector norms, and the number of columns ($r$) in the $U$ and $V$ matrices is allowed to be variable \citep{bach08,bach13,Haeffele:ICML14}.  The norm in \eqref{eq:tensor_norm} has appeared under multiple names in the literature, including the projective tensor norm, decomposition norm, and atomic norm, and by replacing the column norms in \eqref{eq:tensor_norm} with gauge functions the formulation can be generalized to incorporate additional regularization on $(U,V)$, such as non-negativity, while still being a convex function of $X$ \citep{bach13}.  Further, it is worth noting that for particular choices of the $\|\cdot\|_u$ and $\|\cdot\|_v$ vector norms, $\|X\|_{u,v}$ reverts to several well known matrix norms and thus provides a generalization of many commonly used regularizers.  Notably, when the vector norms are both $l_2$ norms, $\|X\|_{2,2} = \|X\|_*$, and the form in \eqref{eq:tensor_norm} is the well known variational definition of the nuclear norm.

The $\|\cdot\|_{u,v}$ norm has the appealing property that by an appropriate choice of vector norms $\|\cdot\|_u$ and $\|\cdot\|_v$ (or more generally gauge functions), one can promote desired properties in the factorized matrices $(U,V)$ while still working with a problem which is convex w.r.t. the product $X=UV^T$.  Based on this concept, several studies have explored optimization problems over factorized matrices $(U,V)$ of the form
\begin{equation}
\label{eq:tensor_norm_obj}
\min_{U,V} \ell(UV^T)+\lambda \|UV^T\|_{u,v}
\end{equation}
Even though the problem is still non-convex w.r.t. the factorized matrices $(U,V)$, it can be shown using ideas from \citet{burer05} on factorized semidefinite programming that, subject to a few general conditions, then local minima of \eqref{eq:tensor_norm_obj} will be global minima  \citep{bach08,Haeffele:ICML14}, which can significantly reduce the dimensionality of some large scale optimization problems.  Unfortunately, aside from a few special cases, the norm defined by \eqref{eq:tensor_norm} (and related regularization functions such as those discussed by \citet{bach13}) cannot be evaluated efficiently, much less optimized over, due to the complicated and non-convex nature of the definition.  As a result, in practice one is often forced to replace \eqref{eq:tensor_norm_obj} by the closely related problem
\begin{equation}
\label{eq:tensor_fact_obj}
\min_{U,V} \ell(UV^T)+\lambda \sum_{i=1}^r \|U_i\|_u \|V_i\|_v = 
\min_{U,V} \ell(UV^T)+\lambda \sum_{i=1}^r \tfrac{1}{2} (\|U_i\|_u^2 + \|V_i\|_v^2)
\end{equation}

However, \eqref{eq:tensor_norm_obj} and \eqref{eq:tensor_fact_obj} are not equivalent problems, due to the fact that solutions to \eqref{eq:tensor_norm_obj} include \textit{any} factorization $(U,V)$ such that their product equals the optimal solution, $UV^T = X_{opt}$, while in \eqref{eq:tensor_norm_obj} one is specifically searching for a factorization $(U,V)$ that achieves the infimum in \eqref{eq:tensor_norm}; in brief, solutions to \eqref{eq:tensor_fact_obj} will be solutions to \eqref{eq:tensor_norm_obj}, but the converse is not true.  As a consequence, results guaranteeing that local minima of the form \eqref{eq:tensor_norm_obj} will be global minima cannot be applied to the formulation in \eqref{eq:tensor_fact_obj}, which is typically more useful in practice.  Here we focus our analysis on the more commonly used family of problems, such as \eqref{eq:tensor_fact_obj}, and show that similar guarantees can be provided regarding the global optimality of local minima.  Additionally, we show that these ideas can be significantly extended to a very wide range of non-convex models and regularization functions, with applications such as tensor factorization and certain forms of neural network training being additional special cases of our framework.

In the context of neural networks, \citet{Bengio:NIPS2005} showed that for neural networks with a single hidden layer, if the number of neurons in the hidden layer is not fixed, but instead fit to the data through a sparsity inducing regularization, then the process of training a globally optimal neural network is analgous to selecting a finite number of hidden units from the infinite dimensional space of all possible hidden units and taking a weighted summation of these units to produce the output. Further, these ideas have very recently been used to analyze the generalization performance of such networks \citep{Bach:2014}.  Here, our results take a similar approach and extend these ideas to certain forms of multi-layer neural networks.  Additionally, our framework provides sufficient conditions on the network architecture to guarantee that from any intialization a globally optimal solution can be found by performing purely local descent on the network weights.

%% file: PosFactor_preliminaries.tex
\section{Preliminaries}

Before we present our main results, we first describe our notation system and recall a few definitions.  

\subsection{Notation}
Our formulation is fairly general in regards to the dimensionality of the data and factorized variables.  As a result, to simplify the notation, we will use capital letters as a shorthand for a set of dimensions, and individual dimensions will be denoted with lower case letters.  For example, $X \in \Re^{d_1 \times \ldots \times d_N} \equiv X \in \Re^D$ for $D = d_1 \times \ldots \times d_N$; we also denote the cardinality of $D$ as $\card(D) = \prod_{i=1}^N d_i$.  Similarly, $X \in \Re^{D \times R} \equiv X \in \Re^{d_1 \times \ldots \times d_N \times r_1 \times \ldots \times r_M}$ for $D=d_1 \times \ldots \times d_N$ and $R=r_1 \times \ldots \times r_M$.
%

Given an element from a tensor space, we will use a subscript to denote a slice of the tensor along the last dimension.  For example, given a matrix $X \in \Re^{d_1 \times r}$, then $X_i \in \Re^d_1, i \in \{1,\ldots,r\}$, denotes the $i$'th column of $X$.  Similarly, given a cube $X \in \Re^{d_1 \times d_2 \times r}$ then $X_i \in \Re^{d_1 \times d_2}, i \in \{1\ldots,r\}$, denotes the $i$'th slice along the third dimension.  Further, given two tensors with matching dimensions except for the last dimension, $X \in \Re^{D \times r_x}$ and $Y \in \Re^{D \times r_y}$, we will use $[X \ Y] \in \Re^{D \times (r_x+r_y)}$ to denote the concatenation of the two tensors along the last dimension. 

We denote the dot product between two elements from a tensor space $(x \in \Re^D, y\in \Re^D)$ as $\left<x,y\right> = vec(x)^T vec(y)$, where $vec(\cdot)$ denotes flattening the tensor into a vector.  For a function $\theta(x)$, we denote its image as $\Image(\theta)$ and its Fenchel dual as $\theta^*(x) \equiv \sup_z \left<x,z\right> - \theta(z)$.  The gradient of a differentiable function $\theta(x)$ is denoted $\nabla \theta(x)$, and the subgradient of a convex (but possibly non-differentiable) function $\theta(x)$ is denoted $\partial \theta(x)$.  For a differentiable function with multiple variables $\theta(x^1,\ldots,x^K)$, we will use $\nabla_{x^i} \theta(x^1,\ldots,x^K)$ to denote the portion of the gradient corresponding to $x^i$. The space of non-negative real numbers is denoted $\Re_+$, and the space of positive integers is denoted $\Nplus$.

\subsection{Definitions}
We now make/recall a few general definitions and well known facts which will be used in our analysis.
\begin{definition}
A \textbf{size-r set of K factors} $(X^1,\ldots,X^K)_r$ is defined to be a set of $K$ tensors where the final dimension of each tensor is equal to $r$.  This is to be interpreted $(X^1,\ldots,X^K)_r \in \Re^{(D^1 \times r)} \times \ldots \times \Re^{(D^K \times r)}$.
\end{definition}
\begin{definition}
The \textbf{indicator function of a set $C$} is defined as
\begin{equation}
\delta_C(x) = \left\{ \begin{array}{cc} 0 & x \in C \\ \infty & x \notin C \end{array} \right.
\end{equation}
\end{definition}
\begin{definition}
A function $\theta : \Re^{D^1} \times \ldots \times \Re^{D^N} \rightarrow \Re^D$ is \textbf{positively homogeneous with degree p} if $\theta(\alpha x^1,\ldots,\alpha x^N) = \alpha^p \theta(x^1,\ldots,x^N), \  \forall \alpha \geq 0$.
\end{definition}
Note that this definition also implies that $\theta(0,\ldots,0) = 0$ for $p\neq 0$.

%
\begin{definition}
A function $\theta : \Re^{D^1} \times \ldots \times \Re^{D^N} \rightarrow \Re_+$ is \textbf{positive semidefinite} if $\theta(0,\ldots,0)=0$ and $\theta(x^1,\ldots,x^N) \geq 0, \ \forall (x^1,\ldots,x^N)$.
\end{definition}
\begin{definition}
The \textbf{one-sided directional derivative} of a function $\theta(x)$ at a point $x$ in the direction $z$ is denoted $\theta(x)(z)$ and defined as
$d\theta(x)(z) \equiv \lim_{\epsilon \searrow 0} \ \ (\theta(x+\epsilon z) - \theta(x)) \epsilon^{-1}$.
\end{definition}
Also, recall that for a differentiable function $\theta(x)$, $d\theta(x)(z) = \left< \nabla \theta(x), z \right>$.

%% file: PosFactor_problem.tex
\section{Problem Formulation}
Returning to the motivating example from the introduction \eqref{eq:intro_tensor_fact_obj}, we now define the family of mapping functions from the factors into the output space and the family of regularization functions on the factors ($\Phi$ and $\Theta$, respectively) which we will study in our framework.  

\subsection{Factorization Mappings}
In this paper, we consider mappings $\Phi$ which are based on a sum of what we refer to as an \textbf{elemental mapping}.  Specifically, if we are given a size-$r$ set of $K$ factors $(X^1,\ldots,X^K)_r$, the elemental mapping $\phi : \Re^{D^1} \times \ldots \times \Re^{D^K} \rightarrow \Re^D$ takes a slice along the last dimension from each tensor in the set of factors and maps it into the output space.  We then define the full mapping to be the sum of these elemental mappings along each of the $r$ slices in the set of factors.  The only requirement we impose on the elemental mapping is that it must be positively homogeneous.  More formally,
\begin{definition}
An \textbf{elemental mapping}, $\phi : \Re^{D^1} \times \ldots \times \Re^{D^K} \rightarrow \Re^D$ is any mapping which is positively homogeneous with degree $p \neq 0$.  The \textbf{r-element factorization mapping} $\Phi_r : \Re^{ (D^1 \times r) } \times \ldots \times \Re^{(D^K \times r) } \rightarrow \Re^D$ is defined as
\begin{equation}
\label{eq:Phi_r_def}
\Phi_r(X^1,\ldots,X^K) = \sum_{i=1}^r \phi(X^1_i,\ldots,X^K_i).
\end{equation}
\end{definition}
As we do not place any restrictions on the elemental mapping, $\phi$, beyond the requirement that it must be positively homogeneous, there are a wide range of problems that can be captured by a mapping with form \eqref{eq:Phi_r_def}. Several example problems which can be placed in this framework include:

\textbf{\textit{Matrix Factorization}}:  The elemental mapping, $\phi : \Re^{d_1} \times \Re^{d_2} \rightarrow \Re^{d_1 \times d_2}$
\begin{equation}
\phi(u,v) = uv^T
\end{equation}
is positively homogeneous with degree 2 and $\Phi_r(U,V) = \sum_{i=1}^r U_i V_i^T = UV^T$ is simply matrix multiplication for matrices with $r$ columns.

\textbf{\textit{Tensor Decomposition - CANDECOMP/PARAFAC (CP)}}: Slightly more generally, the elemental mapping $\phi : \Re^{d_1} \times \ldots \times \Re^{d_K} \rightarrow \Re^{d_1 \times \ldots \times d_K}$
\begin{equation}
\phi(x^1,\ldots,x^K) = x^1 \otimes \cdots \otimes x^K
\end{equation}
(where $\otimes$ denotes the tensor outer product) results in $\Phi_r(X^1,\ldots,X^K)$ being the mapping used in the rank-$r$ CANDECOMP/PARAFAC (CP) tensor decomposition  model \citep{Kolda:SIAMrev2009}.  Further, instead of choosing $\phi$ to be a simple outer product, we can also generalize this to be any multilinear function of the factors $(X^1_i,\ldots,X^K_i)$\footnote{We note that more general tensor decompositions, such as the general form of the Tucker decomposition, do not explicitly fit inside the framework we describe here; however, by using similar arguments to the ones we develop here, it is possible to show analogous results to those we derive in this paper for more general tensor decompositions, which we do not show for clarity of presentation.}.

\textbf{\textit{Neural Networks with Rectified Linear Units (ReLU)}}: Let $\psi^+(x) \equiv \max \{x,0\}$ be the linear rectification function, which is applied element-wise to a tensor $x$ of arbitrary dimension.  Then if we are given a matrix of training data $V \in \Re^{N \times d_1}$, the elemental mapping $\phi(x^1,x^2) : \Re^{d_1} \times \Re^{d_2} \rightarrow \Re^{N \times d_2}$
\begin{equation}
\label{eq:3_layer_phi}
\phi(x^1,x^2) = \psi^+(V x^1) (x^2)^T
\end{equation}
results in a mapping $\Phi_r(X^1,X^2) = \psi^+(V X^1) (X^2)^T$, which can be interpreted as producing the $d_2$ outputs of a 3 layer neural network with $r$ hidden units in response to the input of $N$ data points of $d_1$ dimensional data, $V$.  The hidden units have a ReLU non-linearity; the other units are linear; and the $(X^1,X^2) \in \Re^{d_1 \times r} \times \Re^{d_2 \times r}$ matrices contain the connection weights from the input-to-hidden and hidden-to-output layers, respectively.

By utilizing more complicated definitions of $\phi$, it is possible to consider a broad range of neural network architectures. As a simple example of networks with multiple hidden layers, an elemental mapping such as $\phi : \Re^{d_1 \times d_2} \times \Re^{d_2 \times d_3} \times \Re^{d_3 \times d_4} \times \Re^{d_4 \times d_5} \rightarrow \Re^{N \times d_5}$
\begin{equation}
\phi(x^1,x^2,x^3,x^4) = \psi^+(\psi^+(\psi^+(V x^1) x^2) x^3) x^4
\end{equation}
gives a $\Phi_r(X^1,X^2,X^3,X^4)$ mapping which is the output of a 5 layer neural network in response to the inputs in the $V \in \Re^{N \times d_1}$ matrix with ReLU non-linearities on all of the hidden layer units.  In this case, the network has the architecture that there are $r$, 4 layer fully-connected subnetworks, with each subnetwork having the same number of units in each layer as defined by the dimensions $\{d_2,d_3,d_4\}$.  The $r$ subnetworks are all then fed into a fully connected linear layer to produce the output.

More general still, since \textit{any} positively homogenous transformation is a potential elemental mapping, by an appropriate definition of $\phi$, one can describe neural networks with very general architectures, provided the non-linearities in the network are compatible with positive homogeneity.  Note that max-pooling and rectification are both positively homogeneous and thus fall within our framework.  For example, the well-known ImageNet network from \citep{Krizhevsky:NIPS2012}, which consists of a series of convolutional layers, linear-rectification, max-pooling layers, response normalization layers, and fully connected layers, can be described by taking $r=1$ and defining $\phi$ to be the entire transformation of the network (with the removal of the response normalization layers, which are not positively homogenous).  Note, however, that our results will rely on $r$ potentially changing size or being initialized to be sufficiently large, which limits the applicability of our results to current state-of-the-art network architectures (see discussion).

Here we have provided a few examples of common factorization mappings that can be cast in form \eqref{eq:Phi_r_def}, but certainly there are a wide variety of other problems for which our framework is relevant.  Additionally, while all of the mappings described above are positively homogeneous with degree equal to the degree of the factorization ($K$), this is not a requirement; $p\neq0$ is sufficient.  For example, non-linearities such as a rectification followed by raising each element to a non-zero power are positively homogeneous but of a possibly different degree.  What will turn out to be essential, however, is that we require $p$ to match the degree of positive homogeneity used to regularize the factors, which we will discuss in the next section.

\subsection{Factorization Regularization}

Inspired by the ideas from structured convex matrix factorization, instead of trying to analyze the optimization over a size-$r$ set of $K$ factors $(X^1,\ldots,X^K)_r$ for a fixed $r$, we instead consider the optimization problem where $r$ is possibly allowed to vary and adapted to the data through regularization.  To do so, we will define a regularization function similar to the $\|\cdot\|_{u,v}$ norm discussed in matrix factorization which is convex with respect to the output tensor but which still allows for regularization to be placed on the factors.  Similar to our definition in \eqref{eq:Phi_r_def}, we will begin by first defining an \textbf{elemental regularization function} $g : \Re^{D^1} \times \ldots \times \Re^{D^K} \rightarrow \Re_+ \cup \infty$ which takes as input slices of the factorized tensors along the last dimension and returns a non-negative number.  The requirements we place on $g$ are that it must be positively homogeneous and positive semidefinite.  Formally,
\begin{definition}
\label{def:elemental_reg}
We define an \textbf{elemental regularization function} $g : \Re^{D^1} \times \ldots \times \Re^{D^K} \rightarrow \Re_+ \cup \infty$, to be any function which is positive semidefinite and positively homogeneous.
\end{definition}


Again, due to the generality of the framework, there are a wide variety of possible elemental regularization functions.  We highlight two positive semidefinite, positively homogeneous functions which are commonly used and note that functions can be composed with summations, multiplications, and raising to non-zero powers to change the degree of positive homogeneity and combine various functions.

\textbf{\textit{Norms}}: Any norm $\|x\|$ is positively homogeneous with degree 1.
Note that because we make no requirement of convexity on $g$, this framework can also include functions such as the $l_q$ pseudo-norms for $q \in (0,1)$.

\textbf{\textit{Conic Indicators}}: The indicator function $\delta_C(x)$ of any conic set $C$ is positively homogeneous for all degrees.  Recall that a conic set, $C$, is simply any set such that if $x \in C$ then $\alpha x \in C, \ \forall \alpha \geq 0$. A few popular conic sets which can be of interest include the non-negative orthant $\Re^D_+$, the kernel of a linear operator $\{x : Ax = 0\}$, inequality constraints for a linear operator $\{x: Ax \geq 0\}$, and the set of positive semidefinite matrices.  Constraints on the non-zero support of $x$ are also typically conic sets.  For example, the set $\{ x : \|x\|_0 \leq n \}$ is a conic set, where $\|x\|_0$ is simply the number of non-zero elements in $x$ and $n$ is a positive integer.  More abstractly, conic sets can also be used to enforce invariances w.r.t. positively homogeneous transformations.  For example, given two positively homogeneous functions $\theta(x), \ \theta'(x)$ with equal degrees of positive homogeneity, the sets $\{x : \theta(x) = \theta'(x)\}$ and $\{x : \theta(x) \geq \theta'(x) \}$ are also conic sets.

A few typical formulations of a $g$ which are positively homogeneous with degree $K$ might include:
\begin{align}
	\label{eq:norm_prod_example}
	g(x^1,\ldots,x^K) &= \prod_{i=1}^K \|x^i\|_{(i)} \\
	\label{eq:norm_sum_example}
	g(x^1,\ldots,x^K) &= \tfrac{1}{K}\sum_{i=1}^K \|x^i\|_{(i)}^K \\
	\label{eq:non_seg_example}
	g(x^1,\ldots,x^K) &= \prod_{i=1}^K (\|x^i\|_{(i)} + \delta_{C_i} (x^i))
\end{align}
where all of the norms, $\|\cdot\|_{(i)}$, are arbitrary.  Forms \eqref{eq:norm_prod_example} and \eqref{eq:norm_sum_example} can be shown to be equivalent, in the sense that they give rise to the same $\Omega_{\phi,g}$ function, for all of the example mappings $\phi$ we have discussed here and by an appropriate choice of norm can induce various properties in the factorized elements (such as sparsity), while form \eqref{eq:non_seg_example} is similar but additionally constrains each factor to be an element of a conic set $C_i$ \citep[see][for examples from matrix factorization]{bach08,bach13,Haeffele:ICML14}.

To define our regularization function on the output tensor, $X = \Phi(X^1,\ldots,X^K)$, it will be necessary that the elemental regularization function, $g$, and the elemental mapping, $\phi$, satisfy a few properties to be considered 'compatible' for the definition of our regularization function.  Specifically, we will require the following definition. 
\begin{definition}
\label{def:nondegenerate}
Given an elemental mapping $\phi$ and an elemental regularization function $g$, will we say that $(\phi,g)$ are a \textbf{nondegenerate pair} 
if 1) $g$ and $\phi$ are both positively homogeneous with degree $p$, for some $ p\neq 0$ and 2) $\forall X \in \Image(\phi) \backslash 0, \ \ \exists \mu \in (0,\infty]$ and $(\tilde{z}^1,\ldots,\tilde{z}^K)$ such that $\phi(\tilde{z}^1,\ldots,\tilde{z}^K) = X$, $g(\tilde{z}^1,\ldots,\tilde{z}^K)=\mu$, and $g(z^1,\ldots,z^K) \geq \mu$ for all $(z^1,\ldots,z^K)$ such that $\phi(z^1,\ldots,z^K) = X. $\footnote{Property 1 from the definition of a nondegenerate pair will be critical to our formulation.  Several of our results can be shown without Property 2, but Property 2 is almost always satisfied for most interesting choices of $(\phi,g)$ and is designed to avoid 'pathological' $\Omega_{\phi,g}$ functions (such as $\Omega_{\phi,g} (X) = 0 \ \forall X$).  For example, in matrix factorization with $\phi(u,v) = uv^T$, taking $g(u,v)=\delta_C(u) \|v\|^2$ for any arbitrary norm and conic set $C$ satisfies Property 1 but not Property 2, as we can always reduce the value of $g(u,v)$ by scaling $v$ by a constant $\alpha \in (0,1)$ and scaling $u$ by $\alpha^{-1}$ without changing the value of $\phi(u,v)$.}
\end{definition}

From this, we now define our main regularization function: 
\begin{definition}
Given an elemental mapping $\phi$ and an elemental regularization function $g$ such that $(\phi,g)$ are a nondegenerate pair, we define the \textbf{factorization regularization function}, $\Omega_{\phi,g}(X) : \Re^D \rightarrow \Re_+ \cup \infty$ to be
\begin{equation}
\label{eq:Omega_def}
\begin{split}
\Omega_{\phi,g}(X) \equiv \inf_{r \in \Nplus} \inf_{(X^1,\ldots,X^K)_r} \sum_{i=1}^r g(X^1_i,\ldots,X^K_i) & \\
 \textnormal{subject to } \ \Phi_r(X^1,\ldots,X^K) = X &
\end{split}
\end{equation}
with the additional condition that $\Omega_{\phi,g}(X) = \infty$ if $X \notin \bigcup_r \Image(\Phi_r)$.
\end{definition}

We will show that $\Omega_{\phi,g}$ is a convex function of $X$ and that in general the infimum in \eqref{eq:Omega_def} can always be achieved with a finitely sized factorization (i.e., $r$ does not need to approach $\infty$)\footnote{In particular, the largest $r$ needs to be is $\card(D)$, and we note that $\card(D)$ is a worst case upper bound on the size of the factorization. In certain cases the bound can be shown to be lower.  As an example, $\Omega_{\phi,g}(X) = \|X\|_*$ when $\phi(u,v) = u v^T$ and $g(u,v) = \|u\|_2 \|v\|_2$.  In this case the infimum can be achieved with $r \leq \textnormal{rank}(X) \leq \min \{\card(u),\card(v)\}$.}.  While $\Omega_{\phi,g}$ suffers from many of the practical issues associated with the matrix norm $\|\cdot\|_{u,v}$ discussed earlier (namely that in general it cannot be evaluated in polynomial time due to the complicated definition), because $\Omega_{\phi,g}(X)$ is a convex function on $X$, this allows us to 
use $\Omega_{\phi,g}$ purely as an analysis tool to derive results for a more tractable factorized formulation.

\subsection{Problem Definition}

To build our analysis, we will start by defining the convex (but typically non-tractable) problem, given by
\begin{equation}
\label{eq:convex_obj}
\min_{X,Q} F(X,Q) = \ell(X,Q) + \lambda \Omega_{\phi,g}(X) + H(Q).
\end{equation}
Here $X \in \Re^D$ is the output of the factorization mapping $X = \Phi(X^1,\ldots,X^K)$ as we have been discussing, and the $Q$ term is an optional additional set of non-factorized variables which can be helpful in modeling some problems (for example, to add intercept terms or to model outliers in the data). For our analysis we will assume the following:

\begin{assumption}
	$\ell(X,Q)$ is once differentiable and jointly convex in $(X,Q)$
\end{assumption}
\begin{assumption}
	$H(Q)$ is convex (but possibly non-differentiable)
\end{assumption}
\begin{assumption}
	$(\phi,g)$ are a nondegenerate pair as defined by Definition \ref{def:nondegenerate}; $\Omega_{\phi,g}(X)$ is as defined by \eqref{eq:Omega_def}; and $\lambda > 0$
\end{assumption}
\begin{assumption}
  The minimum of $F(X,Q)$ exists $\implies \emptyset \neq \argmin_{X,Q} F(X,Q)$.
\end{assumption}

As noted above, it is typically impractical to optimize over functions involving $\Omega_{\phi,g}(X)$, and, even if one were given an optimal solution to \eqref{eq:convex_obj}, $X_{opt}$, one would still need to solve the problem given in \eqref{eq:Omega_def} to recover the desired $(X^1,\ldots,X^K)$ factors.  Therefore, we use \eqref{eq:convex_obj} merely as analysis tool and instead tailor our results to the non-convex optimization problem given by
\begin{equation}
\label{eq:factor_obj}
\begin{split}
\min_{(X^1,\ldots,X^K)_r,Q} &f_r(X^1,\ldots,X^K,Q) \equiv \\
&\ell(\Phi_r(X^1,\ldots,X^K),Q) + \lambda \sum_{i=1}^r g(X^1_i,\ldots,X^K_i) + H(Q).
\end{split}
\end{equation}
We will show in the next section that any local minima of \eqref{eq:factor_obj} is a global minima if it satisfies the condition that one slice from each of the factorized tensors is all zero.  Further, we will also show that if $r$ is taken to be large enough then from any initialization we can always find a global minimum of \eqref{eq:factor_obj} by doing an optimization based purely on local descent.

%% file: PosFactor_mainresults.tex
\section{Main Analysis}

We begin our analysis by first showing a few simple properties and lemmas relevant to our framework. 

\subsection{Preliminary Results}

First, from the definition of $\Phi_r$ it is easy to verify that if $\phi$ is positively homogeneous with degree $p$, then $\Phi_r$ is also positively homogeneous with degree $p$ and satisfies the following proposition
\begin{proposition}
\label{prop:Phi_r_prop}
Given a size-$r_x$ set of $K$ factors, $(X^1,\ldots,X^K)_{r_x}$, and a size-$r_y$ set of $K$ factors, $(Y^1,\ldots,Y^K)_{r_y}$, then $\forall \alpha \geq 0, \beta \geq 0$
\begin{equation}
\Phi_{(r_x+r_y)}([\alpha X^1 \, \beta Y^1],\ldots,[\alpha X^K \, \beta Y^K]) = 
\alpha^p \Phi_{r_x}(X^1,\ldots,X^K) + \beta^p \Phi_{r_y}(Y^1,\ldots,Y^K)
\end{equation}
where recall, $[X \ Y]$ denotes the concatenation of $X$ and $Y$ along the final dimension of the tensor.
\end{proposition}
Further, $\Omega_{\phi,g}(X)$ satisfies the following proposition:
\begin{proposition}
\label{prop:Omega_properties}
The function $\Omega_{\phi,g} : \Re^D \rightarrow \Re \cup \infty$ as defined in \eqref{eq:Omega_def} has the properties
\begin{enumerate}
\item $\Omega_{\phi,g}(0) = 0$ and $\Omega_{\phi,g}(X) > 0 \ \ \forall X \neq 0$.
\item $\Omega_{\phi,g}$ is positively homogeneous with degree 1.
\item $\Omega_{\phi,g}(X+Y) \leq \Omega_{\phi,g}(X)+\Omega_{\phi,g}(Y) \ \ \forall (X,Y)$ 
\item $\Omega_{\phi,g}(X)$ is convex w.r.t. $X \in \Re^D$.
\item The infimum in \eqref{eq:Omega_def} can be achieved with $r \leq \card(D) \ \ \forall X$ s.t. $\Omega_{\phi,g}(X) < \infty$.
\end{enumerate}
\end{proposition}
\begin{proof} \textbf{Proposition \ref{prop:Omega_properties}}
Many of these properties can be shown in a similar fashion to results from the $\|\cdot\|_{u,v}$ norm discussed previously \citep{bach08,Yu:2014}.

1) By definition and the fact that $g$ is positive semidefinite, we always have $\Omega_{\phi,g}(X) \geq 0 \ \ \forall X$.  Trivially, $\Omega_{\phi,g}(0) = 0$ since we can always take $(X^1,\ldots,X^K) = (0,\ldots,0)$ to achieve the infimum.  For $X \neq 0$, because $(\phi,g)$ is a non-degenerate pair then $\sum_{i=1}^r g(X^1_i,\ldots,X^K_i) > 0$ for any $(X^1,\ldots,X^K)_r \ \ \textnormal{s.t.} \ \ \Phi_r(X^1,\ldots,X^K) = X$ and $r$ finite.  Property 5) shows that the infimum can be achieved with $r$ finite, completing the result.

2) For all $\alpha \geq 0$ and any $(X^1,\ldots,X^K)_r$ such that $X = \Phi_r(X^1,\ldots,X^K)$, note that from positive homogeneity $\Phi_r(\alpha^{1/p} X^1,\ldots,\alpha^{1/p} X^K) = \alpha X$ and $\sum_{i=1}^r g(\alpha^{1/p} X_i^1,\ldots,\alpha^{1/p} X_i^K) = \alpha \sum_{i=1}^r g(X^1_i,\ldots,X^K_i)$.  Applying this fact to the definition of $\Omega_{\phi,g}$ gives that $\Omega_{\phi,g}(\alpha X) = \alpha \Omega_{\phi,g,}(X)$. 

3) If either $\Omega_{\phi,g}(X) = \infty$ or $\Omega_{\phi,g}(Y)=\infty$ then the inequality is trivially satisfied.  Considering any $(X,Y)$ pair such that $\Omega_{\phi,g}$ is finite for both $X$ and $Y$, for any $\epsilon > 0$ let $(X^1,\ldots,X^K)_{r_x}$ be an $\epsilon$ optimal factorization of $X$.  Specifically, $\Phi_{r_x}(X^1,\ldots,X^K) = X$ and $\sum_{i=1}^{r_x} g(X^1_i,\ldots,X^K_i) \leq \Omega_{\phi,g}(X) + \epsilon$.  Similarly, let $(Y^1,\ldots,Y^K)_{r_y}$ be an $\epsilon$ optimal factorization of $Y$.  From Proposition \ref{prop:Phi_r_prop} we have $\Phi_{r_x+r_y}([X^1 \, Y^1],\ldots,[X^K \, Y^K]) = X+Y$, so $\Omega_{\phi,g}(X+Y) \leq \sum_{i=1}^{r_x} g(X^1_i,\ldots,X^K_i) + \sum_{j=1}^{r_y} g(Y^1_j,\ldots,X^K_j) \leq \Omega_{\phi,g}(X) + \Omega_{\phi,g}(Y) + 2\epsilon$.  Letting $\epsilon$ tend to 0 completes the result.

4) Convexity is given by the combination of properties 2 and 3.  Further, note that properties 2 and 3 also show that $\{X\in \Re^D : \Omega_{\phi,g}(X) < \infty \}$ is a convex set.

5) Let $\Gamma \subset \Re^D$ be defined as 
\begin{equation}
\Gamma = \{ X : \exists(x^1,\ldots,x^K), \ \phi(x^1,\ldots,x^K)=X, \ g(x^1,\ldots,x^K) \leq 1 \}
\end{equation}
Note that because $(\phi,g)$ is a nondegenerate pair, for any non-zero $X \in \Gamma$ there exists $\alpha \in [1,\infty)$ such that $\alpha X$ is on the boundary of $\Gamma$, so $\Gamma$ and its convex hull are compact sets. 

Further, note that $\Gamma$ contains the origin by definition of $\phi$ and $g$, so as a result, $\Omega_{\phi,g}$ is equivalent to a gauge function on the convex hull of $\Gamma$ 
\begin{equation}
\Omega_{\phi,g}(X) = \inf_{\mu} \{\mu : \mu \geq 0, \ X \in \mu \ \textnormal{conv} (\Gamma) \} \\
\end{equation}
Since the infimum w.r.t. $\mu$ is linear and constrained to a compact set, it must be achieved.  Therefore, there must exist $\mu_{opt} \geq 0$, $\{ \theta \in \Re^{\card(D)} : \theta_i \geq 0 \ \forall i, \ \sum_{i=1}^{\card(D)} \theta_i = 1 \}$, and $\{(Z_i^1,\ldots,Z_i^K) : \phi(Z^1_i,\ldots,Z^K_i) \in \Gamma \}_{i=1}^{\card(D)}$ such that $X = \mu_{opt} \sum_{i=1}^{\card(D)} \theta_i \phi(Z^1_i,\ldots,Z^K_i)$ and $\Omega_{\phi,g}(X) = \mu_{opt}$.

This, combined with positive homogeneity, completes the result as we can take $(X_i^1,\ldots,X_i^K) = ((\mu_{opt} \theta_i)^{1/p}Z^1_i,\ldots,(\mu_{opt} \theta_i)^{1/p}Z^K_i)$, which gives
\begin{equation}
\mu_{opt} = \Omega_{\phi,g}(X) \leq \sum_{i=1}^{\card(D)} g(X^1_i,\ldots,X^K_i) = \mu_{opt} \sum_{i=1}^{\card(D)} \theta_i g(Z^1_i,\ldots,Z^K_i) \leq \mu_{opt} \sum_{i=1}^{\card(D)} \theta_i = \mu_{opt}
\end{equation}
and shows that a factorization of size-$\card(D)$ which achieves the infimum must exist.
\end{proof}

We next derive the Fenchel dual of $\Omega_{\phi,g}$, which will provide a useful characterization of the subgradient of $\Omega_{\phi,g}$.

\begin{proposition}
The Fenchel dual of $\Omega_{\phi,g}(X)$ is given by
\begin{equation}
\Omega^*_{\phi,g}(W) = \left\{\begin{array}{cc} 0 & \Omega^\circ_{\phi,g}(W) \leq 1 \\ \infty & \textnormal{otherwise} \end{array} \right.
\end{equation}
where
\begin{equation}
\begin{split}
\label{eq:polar}
\Omega^\circ_{\phi,g}(W) \equiv \sup_{(z^1,\ldots,z^K)} \left<W,\phi(z^1,\ldots,z^K)\right> &\\
 \textnormal{subject to } g(z^1,\ldots,z^K) \leq 1 &
\end{split}
\end{equation}
\end{proposition}
\begin{proof}
Recall, $\Omega^*_{\phi,g}(W) \equiv \sup_Z \left<W,Z\right>-\Omega_{\phi,g}(Z)$, so for $Z$ to approach the supremum we must have $Z \in \bigcup_r \Image(\Phi_r)$. As result, the problem is equivalent to
\begin{align}
\Omega^*_{\phi,g}(W) &= \sup_{r \in \Nplus} \sup_{(Z^1,\ldots,Z^K)_r} \left<W,\Phi_r(Z^1,\ldots,Z^K)\right> - \sum_{i=1}^r g(Z^1_i,\ldots,Z^K_i) \\
\label{eq:fenchel_sum}
&= \sup_{r \in \Nplus} \sup_{(Z^1,\ldots,Z^K)_r} \sum_{i=1}^r \left[\left<W,\phi(Z^1_i,\ldots,Z^K_i)\right> - g(Z^1_i,\ldots,Z^K_i) \right]
\end{align}
If $\Omega^\circ_{\phi,g}(W) \leq 1$ then all the terms in the summation of \eqref{eq:fenchel_sum} will be non-positive, so taking $(Z^1,\ldots,Z^K)=(0,\ldots,0)$ will achieve the supremum.  Conversely, if $\Omega^\circ_{\phi,g}(W) > 1$, then $\exists (z^1,\ldots,z^K)$ such that $\left<W,\phi(z^1,\ldots,z^K)\right> > g(z^1,\ldots,z^K)$.  This result, combined with the positive homogeneity of $\phi$ and $g$ gives that \eqref{eq:fenchel_sum} is unbounded by considering $(\alpha z^1,\ldots,\alpha z^K)$ as $\alpha \rightarrow \infty$.
\end{proof}

We briefly note that the optimization problem associated with \eqref{eq:polar} is typically referred to as the polar problem and is a generalization of the concept of a dual norm.  In practice solving the polar can still be very challenging and is often the limiting factor in applying our results in practice \citep[see][for further information]{bach13,Zhang:NIPS2013}.

With the above derivation of the Fenchel dual, we now recall that if $\Omega_{\phi,g}(X) < \infty$ then the subgradient of $\Omega_{\phi,g}(X)$ can be characterized by $\partial \Omega_{\phi,g} (X) = \{ W : \left< X,W \right> = \Omega_{\phi,g}(X) + \Omega^*_{\phi,g}(W) \}$.  This forms the basis for the following lemma which will be used in our main results
\begin{lemma}
	\label{lem:Omega_opt}
	Given a factorization $X = \Phi_r(X^1,\ldots,X^K)$ and a regularization function $\Omega_{\phi,g}(X)$, then the following conditions are equivalent:
	\begin{enumerate}
		\item $(X^1,\ldots,X^K)$ is an optimal factorization of $X$; i.e., $\sum_{i=1}^r g(X^1_i,\ldots,X^K_i) = \Omega_{\phi,g}(X)$
		\item $\exists W$ such that $\Omega^\circ_{\phi,g}(W) \leq 1$ and $\left< W,\Phi_r(X^1,\ldots,X^K) \right> = \sum_{i=1}^r g(X^1_i,\ldots,X^K_i)$,
		\item $\exists W$ such that $\Omega^\circ_{\phi,g}(W) \leq 1$ and $\forall i \in \{1,\ldots,r\}$, $\left<W,\phi(X^1_i,\ldots,X^K_i)\right> = g(X^1_i,\ldots,X^K_i)$
	\end{enumerate}
	Further, any $W$ which satisfies condition 2 or 3 satisfies both conditions 2 and 3 and $W \in \partial \Omega_{\phi,g}(X)$. 
\end{lemma}
\begin{proof}
2 $\iff$ 3) 3 trivially implies 2 from the definition of $\Phi_r$. For the opposite direction, because $\Omega^\circ_{\phi,g}(W) \leq 1$ we have $\left<W,\phi(X^1_i,\ldots,X^K_i)\right> \leq g(X^1_i,\ldots,X^K_i) \ \ \forall i$. Taking the sum over $i$, we can only achieve equality in 2 if we have equality $\forall i$ in condition 3.  This also shows that any $W$ which satisfies condition 2 or 3 must also satisfy the other condition.

We next show that if $W$ satisfies conditions 2/3 then $W\in \partial \Omega_{\phi,g}(X)$.  First, from condition 2/3 and the definition of $\Omega_{\phi,g}$, we have $\Omega_{\phi,g}(X) \leq \sum_{i=1}^r g(X^1_i,\ldots,X^K_i) = \left<W,X\right> < \infty$.  Thus, recall that because $\Omega_{\phi,g}(X)$ is convex and finite at $X$, we have $\left< W,X \right> \leq \Omega_{\phi,g}(X) + \Omega^*_{\phi,g}(W)$ with equality iff $W \in \partial \Omega_{\phi,g}(X)$.  Now, by contradiction assume $W$ satisfies conditions 2/3 but $W \notin \partial \Omega_{\phi,g}(X)$.  From condition 2/3 we have $\Omega^*_{\phi,g}(W)=0$, so $\Omega_{\phi,g}(X)=\Omega_{\phi,g}(X)+\Omega^*_{\phi,g}(W) > \left< X,W \right> = \sum_{i=1}^r g(X^1_i,\ldots,X^K_i)$ which contradicts the definition of $\Omega_{\phi,g}(X)$.

1 $\implies$ 2) Any $W \in \partial \Omega_{\phi,g}(X)$ satisfies $\left< X,W \right> = \Omega_{\phi,g}(X) + \Omega^*_{\phi,g}(W) = \sum_{i=1}^r g(X^1_i,\ldots,X^K_i)$.
	
2 $\implies$ 1) By contradiction, assume $(X^1,\ldots,X^K)_r$ was not an optimal factorization of $X$.  This gives, $\Omega_{\phi,g}(X) < \sum_{i=1}^r g(X^1_i,\ldots,X^K_i) = \left< W,X \right> = \Omega_{\phi,g}(X) + \Omega^*_{\phi,g}(W) = \Omega_{\phi,g}(X)$, producing the contradiction.
\end{proof}

Finally, we show one additional lemma before presenting our main results.

\begin{lemma}
\label{lem:theta_search}

If $(X^1,\ldots,X^K,Q)$ is a local minimum of $f_r(X^1,\ldots,X^K,Q)$ as given in \eqref{eq:factor_obj}, then for any $\theta \in \Re^r$
\begin{equation}
\left< -\tfrac{1}{\lambda} \nabla_X \ell(\Phi_r(X^1,\ldots,X^K),Q),\sum_{i=1}^r \theta_i \phi(X^1_i,\ldots,X^K_i) \right> = \sum_{i=1}^r \theta_i g(X^1_i,\ldots,X^K_i)
\end{equation}

\end{lemma}
\begin{proof}
Let $(Z^1_i,\ldots,Z^K_i) = (\theta_i X^1_i,\ldots,\theta_i X^K_i)$ for all $i \in \{1 \ldots r\}$ and let $\Lambda = \sum_{i=1}^r \theta_i \phi(X^1_i,\ldots,X^K_i)$.  From positive homogeneity and the fact that we have a local minimum, then $\exists \delta > 0$ such that $\forall \epsilon \in (0,\delta)$ we must have
\begin{align}
&f_r(X^1,\ldots,X^K,Q) \leq f_r(X^1+\epsilon Z^1,\ldots,X^K+\epsilon Z^K,Q) \implies \\
\begin{split}
\label{eq:lem_localmin_inequal}
&\ell(\Phi_r(X^1,\ldots,X^K),Q) + \lambda \sum_{i=1}^r g(X^1_i,\ldots,X^K_i) + H(Q) \leq \\
&\ell \left(\sum_{i=1}^r (1+\epsilon \theta_i)^{p} \phi(X^1_i,\ldots,X^K_i),Q \right) + \lambda \sum_{i=1}^r (1+\epsilon \theta_i)^{p} g(X^1_i,\ldots,X^K_i) + H(Q) 
\end{split}
\end{align}
Taking the first order approximation $(1+\epsilon \theta_i)^{p} = 1+p \epsilon \theta_i+ O(\epsilon^2)$ and rearranging the terms of \eqref{eq:lem_localmin_inequal}, we arrive at
\begin{align}
\begin{split}
\label{eq:pos_epsilon}
0 \leq &\ell \left(\Phi_r(X^1,\ldots,X^K) + p \epsilon \Lambda + O(\epsilon^2),Q \right) - 
\ell(\Phi_r(X^1,\ldots,X^K),Q) \\&+ p \epsilon \lambda \sum_{i=1}^r \theta_i g(X^1_i,\ldots,X^K_i) + O(\epsilon^2)
\end{split}
\end{align}
Taking $\lim_{\epsilon \searrow 0} [\tfrac{\eqref{eq:pos_epsilon}}{\epsilon}]$, we note that the difference in the $\ell(\cdot,\cdot)$ terms gives the one-sided directional derivative $d\ell(\Phi_r(X^1,\ldots,X^K),Q)(p \Lambda,0)$, thus from the differentiability of $\ell$ we get
\begin{align}
\label{eq:dot_lem_dir1}
0 \leq \left<\nabla_X \ell(\Phi_r(X^1,\ldots,X^K),Q),p \Lambda \right> + p \lambda \sum_{i=1}^r \theta_i g(X^1_i,\ldots,X^K_i)
\end{align}
Noting that for $\epsilon > 0$ but sufficiently small, we also must have $f_r(X^1,\ldots,X^K,Q) \leq f_r(X^1 - \epsilon Z^1,\ldots,X^K - \epsilon Z^K)$, using identical steps as before and taking the first order approximation $(1-\epsilon \theta_i)^{p} = 1-p \epsilon \theta_i+ O(\epsilon^2)$, we get
\begin{align}
\begin{split}
\label{eq:neg_epsilon}
0 \leq &\ell (\Phi_r(X^1,\ldots,X^K) - p \epsilon \Lambda + O(\epsilon^2),Q) - 
\ell(\Phi_r(X^1,\ldots,X^K),Q) \\ &- p \epsilon \lambda \sum_{i=1}^r \theta_i g(X^1_i,\ldots,X^K_i) + O(\epsilon^2)
\end{split}
\end{align}
and taking the limit $\lim_{\epsilon \searrow 0} [\tfrac{\eqref{eq:neg_epsilon}}{\epsilon}]$, we arrive at
\begin{align}
\label{eq:dot_lem_dir2}
0 \leq \left<\nabla_X \ell(\Phi_r(X^1,\ldots,X^K),Q), - p \Lambda \right> - p \lambda \sum_{i=1}^r \theta_i g(X^1_i,\ldots,X^K_i)
\end{align}
Combining \eqref{eq:dot_lem_dir1} and \eqref{eq:dot_lem_dir2} and rearranging terms gives the result.
\end{proof}

\subsection{Main Results}

Based on the above preliminary results, we are now ready to state our main results and several immediate corollaries.

\begin{theorem}
\label{thm:0_local_min}

Given a function $f_r(X^1,\ldots,X^K,Q)$ of the form given in \eqref{eq:factor_obj}, any local minimizer of the optimization problem 
\begin{align}
	\begin{split}
	\min_{(X^1,\ldots,X^K)_r,Q} \ &f_r(X^1,\ldots,X^K,Q) \equiv \\ &\ell(\Phi_r(X^1,\ldots,X^K),Q)+\lambda \sum_{i=1}^r g(X^1_i,\ldots,X^K_i) + H(Q)
	\end{split}
\end{align}
such that $(X^1_{i_0},\ldots,X^K_{i_0})=(0,\ldots,0)$ for some $i_0 \in \{1,\ldots,r\}$ is a global minimizer.
\end{theorem}

\begin{proof}{\textbf{Theorem  \ref{thm:0_local_min}}}

We begin by noting that from the definition of $\Omega_{\phi,g}(X)$, for any factorization $X=\Phi_r(X^1,\ldots,X^K)$ 
\begin{equation}
	\label{eq:fact_upper_bound}
	\begin{split}
	&F(X,Q) = \ell(X,Q)+\lambda \Omega_{\phi,g}(X)+H(Q) \leq \\
	&\ell(\Phi_r(X^1,\ldots,X^K),Q)+\lambda \sum_{i=1}^r g(X^1_i,\ldots,X^K_i) + H(Q) = f_r(X^1,\ldots,X^K,Q)
	\end{split}
\end{equation}
with equality at any factorization which achieves the infimum in \eqref{eq:Omega_def}.  We will show that a local minimum of $f_r(X^1,\ldots,X^K,Q)$ satisfying the conditions of the theorem also satisfies the conditions for $(\Phi_r(X^1,\ldots,X^K),Q)$ to be a global minimum of the convex function $F(X,Q)$, which implies a global minimum of $f_r(X^1,\ldots,X^K,Q)$ due to the global bound in \eqref{eq:fact_upper_bound}.

First, because \eqref{eq:convex_obj} is a convex function, a simple subgradient condition gives that $(X,Q)$ is a global minimum of $F(X,Q)$ iff the following two conditions are satisfied
\begin{align}
	\label{eq:subgrad_cond1}
	-\tfrac{1}{\lambda} \nabla_X \ell(X,Q) &\in \partial \Omega_{\phi,g}(X) \\
	\label{eq:subgrad_cond2}
	-\nabla_Q \ell(X,Q) &\in \partial H(Q)
\end{align}
where $\nabla_X \ell(X,Q)$ and $\nabla_Q \ell(X,Q)$ denote the portions of the gradient of $\ell(X,Q)$ corresponding to $X$ and $Q$, respectively.  If $(X^1,\ldots,X^K,Q)$ is a local minimum of $f_r(X^1,\ldots,X^K,Q)$, then \eqref{eq:subgrad_cond2} must be satisfied at $(X,Q) = (\Phi_r(X^1,\ldots,X^K),Q)$, as this is implied by the first order optimality condition for a local minimum \citep[Chap. 10]{Rockafellar:1998}, so we are left to show that \eqref{eq:subgrad_cond1} is also satisfied.

Turning to the factorization objective, if $(X^1,\ldots,X^K,Q)$ is a local minimum of $f_r(X^1,\ldots,X^K,Q)$, then $\forall (Z^1,\ldots,Z^K)_r$ there exists $\delta>0$ such that $\forall \epsilon \in (0,\delta)$ we have $f_r(X^1+\epsilon^{1/p} Z^1,\ldots,X^K+\epsilon^{1/p} Z^K,Q) \geq f_r(X^1,\ldots,X^K,Q)$.  If we now consider search directions $(Z^1,\ldots,Z^K)_r$ of the form
\begin{equation}
	\label{eq:z_form}
	(Z^1_j,\ldots,Z^K_j) = \left\{ \begin{array}{cc} (0,\ldots,0) & j \neq i_0 \\ (z^1,\ldots,z^K) & j = i_0 \end{array} \right.
\end{equation}
where $i_0$ is the index such that $(X^1_{i_0},\ldots,X^K_{i_0})=(0,\ldots,0)$, then for $\epsilon \in (0,\delta)$, we have
\begin{align}
	&\ell(\Phi_r(X^1,\ldots,X^K),Q) + \lambda \sum_{i=1}^r g(X^1_i,\ldots,X^K_i) + H(Q) \leq \\
	\begin{split}
		\label{eq:z_search_1}
		&\ell(\Phi_r(X^1 + \epsilon^{1/p} Z^1,\ldots,X^K+\epsilon^{1/p} Z^K),Q) + \\ 
		&\lambda \sum_{i=1}^r g(X^1_i+\epsilon^{1/p} Z_i^1,\ldots,X_i^K+\epsilon^{1/p} Z_i^K) + H(Q) = 
	\end{split} \\ 
	\begin{split}
		\label{eq:z_search_2}
		&\ell(\Phi_r(X^1,\ldots,X^K)+\epsilon \phi(z^1,\ldots,z^K),Q)+ \\
		&\lambda \sum_{i=1}^r g(X^1_i,\ldots,X^K_i) + \epsilon \lambda g(z^1,\ldots,z^K) + H(Q).
	\end{split}
\end{align}
The equality between \eqref{eq:z_search_1} and \eqref{eq:z_search_2} comes from the special form of $Z$ given by \eqref{eq:z_form} and the positive homogeneity of $\phi$ and $g$. Rearranging terms, we now have
\begin{equation}
	\label{eq:lim_frac1}
	\begin{split}
	\epsilon^{-1} &[ \ell(\Phi_r(X^1,\ldots,X^K)+\epsilon \phi(z^1,\ldots,z^K),Q) -\ell(\Phi_r(X^1,\ldots,X^K),Q) ] \geq \\
	&-\lambda g(z^1,\ldots,z^K).
	\end{split}
\end{equation}
Taking the limit $\lim \epsilon \searrow 0$ of \eqref{eq:lim_frac1}, we note that the left side of the inequality is simply the definition of the one-sided directional derivative of $\ell(\Phi_r(X^1,\ldots,X^K),Q)$ in the direction $(\phi(z^1,\ldots,z^K),0)$, which combined with the differentiability of $\ell(X,Q)$, gives
\begin{align}
	\left<\phi(z^1,\ldots,z^K),\nabla_X \ell(\Phi_r(X^1,\ldots,X^K),Q) \right> \geq -\lambda g(z^1,\ldots,z^K).
\end{align}
Because $(z^1,\ldots,z^K)$ was arbitrary, we have established that
\begin{equation}
	\label{eq:dot_lessthan}
	\begin{split}
	\left<\phi(z^1,\ldots,z^K),-\tfrac{1}{\lambda} \nabla_X \ell(\Phi_r(X^1,\ldots,X^K),Q)\right> &\leq g(z^1,\ldots,z^K) \ \ \forall(z^1,\ldots,z^K) \iff \\
	\Omega_{\phi,g}^\circ(-\tfrac{1}{\lambda} \nabla_X \ell(\Phi_r(X^1,\ldots,X^K),Q)) &\leq 1
	\end{split}
\end{equation}

Further, if we choose $\theta$ to be vector of all ones in Lemma \ref{lem:theta_search}, we get
\begin{align}
	\label{eq:dot_equal}
	\sum_{i=1}^r g(X^1_i,\ldots,X^K_i) = \left<\Phi_r(X^1,\ldots,X^K),-\tfrac{1}{\lambda}\nabla_X \ell(\Phi_r(X^1,\ldots,X^K),Q) \right>
\end{align}
which, combined with \eqref{eq:dot_lessthan} and Lemma \ref{lem:Omega_opt}, shows that $-\tfrac{1}{\lambda} \nabla_X \ell(\Phi_r(X^1,\ldots,X^K),Q) \in \partial \Omega_{\phi,g}(\Phi_r(X^1,\ldots,X^K))$, completing the result.
\end{proof}

From this result, we can then test the global optimality of any local minimum (regardless of whether it has an all-zero slice or not) from the immediate corollary: 

\begin{corollary}
\label{corr:r1}
Given a function $f_r(X^1,\ldots,X^K,Q)$ of the form given in \eqref{eq:factor_obj}, any local minimizer of the optimization problem
\begin{align}
	&\min_{(X^1,\ldots,X^K)_r,Q} f_r(X^1,\ldots,X^K,Q) 
\end{align}
is a global minimizer if $f_{r+1}([X^1 \ 0],\ldots,[X^K \ 0],Q)$ is a local minimizer of $f_{r+1}$.
\end{corollary} 

From the results of Theorem \ref{thm:0_local_min}, we are now also able to show that if we let the size of the factorized variables ($r$) become large enough, then from any initialization we can always find a global minimizer of $f_r(X^1,\ldots,X^K,Q)$ using a purely local descent strategy.  Specifically, we have the following result.


\begin{theorem}
\label{thm:non_inc_path}
Given a function $f_r(X^1,\ldots,X^K,Q)$ as defined by \eqref{eq:factor_obj}, if $r > \card(D)$ then from any point $(Z^1,\ldots,Z^K,Q)$ such that $f_r(Z^1,\ldots,Z^K,Q) < \infty$ there must exist a non-increasing path from $(Z^1,\ldots,Z^K,Q)$ to a global minimizer of $f_r(X^1,\ldots,X^K,Q)$.
\end{theorem}

\begin{proof} \textbf{Theorem \ref{thm:non_inc_path}}

Clearly if $(Z^1,\ldots,Z^K,Q)$ is not a local minimum, then we can follow a decreasing path until we reach a local minimum.  Having arrived at a local minimum, $(\tilde{X}^1,\ldots,\tilde{X}^K,\tilde{Q})$, if $(\tilde{X}^1_i,\ldots,\tilde{X}^K_i)=(0,\ldots,0)$ for any $i \in \{1,\ldots,r\}$ then from Theorem \ref{thm:0_local_min} we must be at a global minimum.  We are left to show that a non-increasing path to a global minimizer must exist from any local minima such that $(\tilde{X}^1_i,\ldots,\tilde{X}^K_i) \neq (0,\ldots,0)$ for all $i \in \{1,\ldots,r\}$.

Let us define the set $S = \{ \sum_{i=1}^r \theta_i \phi(\tilde{X}^1_i,\ldots,\tilde{X}^K_i) : \theta \in \Re^r \}$.  Because $r > \card(D)$ there must exist $\hat{\theta} \in \Re^r$ such that $\hat{\theta} \neq 0$ and $\sum_{i=1}^r \hat{\theta_i} \phi(\tilde{X}^1_i,\ldots,\tilde{X}^K_i) = 0$.  Further, from Lemma \ref{lem:theta_search} we must have that $\sum_{i=1}^r \hat{\theta}_i g(\tilde{X}^1_i,\ldots,\tilde{X}^K_i) = \left<-\tfrac{1}{\lambda}\nabla_X \ell(\Phi_r(X^1,\ldots,X^K),Q), \sum_{i=1}^r \hat{\theta_i} \phi(X^1_i,\ldots,X^K_i) \right> = 0 $.  Because $g(\tilde{X}^1_i,\ldots,\tilde{X}^K_i) > 0, \ \ \forall i \in \{1,\ldots,r\}$ this implies that at least one entry of $\hat{\theta}$ must be strictly less than zero.

Without loss of generality, scale $\hat{\theta}$ so that $\min_i \hat{\theta}_i = -1$.  Now, for all $(\gamma,i) \in \{ [0,1] \} \times \{1,\ldots,r\}$, let us define
\begin{equation}
	(R^1_i(\gamma),\ldots,R^K_i(\gamma)) \equiv ((1+\gamma \hat{\theta_i})^{1/p} \tilde{X}^1_i,\ldots,(1+\gamma \hat{\theta_i})^{1/p} \tilde{X}^K_i)
\end{equation}
where $p$ is the degree of positive homogeneity of $(\phi,g)$. Note that by construction $(R^1(0),\ldots,R^K(0)) = (\tilde{X}^1,\ldots,\tilde{X}^K)$ and that for $\gamma=1$ there must exist $i_0 \in \{1,\ldots,r\}$ such that $(R^1_{i_0}(1),\ldots,R^K_{i_0}(1)) = (0,\ldots,0)$.

Further, from the positive homogeneity of $(\phi,g)$ we have $\forall \gamma \in [0,1]$
\begin{align}
	\begin{split}
	\label{eq:RX_path1}
	f_r(R^1(\gamma),\ldots,R^K(\gamma),Q) =& 
	\ell \left(\sum_{i=1}^r \phi(X_i^1,\ldots,X_i^K) + \gamma \sum_{i=1}^r \hat{\theta}_i \phi(X^1_i,\ldots,X^K_i),\, Q \right) +\\
	&\lambda \gamma \sum_{i=1}^r \hat{\theta}_i g(X^1_i,\ldots,X^K_i) + \lambda \sum_{i=1}^r g(X^1_i,\ldots,X^K_i) + H(Q)
	\end{split} \\
	\label{eq:RX_path2}
	=& \ell(\Phi_r(X^1,\ldots,X^K),Q) + \lambda \sum_{i=1}^r g(X^1_i,\ldots,X^K_i) + H(Q) \\
	=& f_r(X^1,\ldots,X^K,Q) 
\end{align}
where the equality between \eqref{eq:RX_path1} and \eqref{eq:RX_path2} is seen by recalling that $\sum_{i=1}^r \hat{\theta}_i \phi(X^1_i,\ldots,X^K_i) = 0$ and $\sum_{i=1}^r \hat{\theta}_i g(X^1_i,\ldots,X^K_i) =0 $.

As a result, as $\gamma$ goes from $0 \rightarrow 1$ we can traverse a path from $(X^1,\ldots,X^K,Q) \rightarrow (R^1(1),\ldots,R^K(1),Q)$ without changing the value of $f_r$.  Also recall that by construction $(R^1_{i_0}(1),\ldots,R^K_{i_0}(1)) = (0,\ldots,0)$, so if $(R^1(1),\ldots,R^K(1),Q)$ is a local minimizer of $f_r$ then it must be a global minimizer due to Theorem \ref{thm:0_local_min}.  If $(R^1(1),\ldots,R^K(1),Q)$ is not a local minimizer then there must exist a descent direction and we can iteratively apply this result until we reach a global minimizer, completing the proof.
\end{proof}

We note that from this proof we have also described a meta-algorithm (outlined in Algorithm \ref{algorithm}) which can be used with any local-descent optimization strategy to guarantee convergence to a global minimum.  While in general the size of the factorization ($r$) might increase as the algorithm proceeds, as a worst case, it is guaranteed that a global minimum can be found with a finite $r$ never growing larger than $\card(D)+1$.  Also note that this is a worst case upper bound on $r$ for the most general form of our framework and that for specific choices of $\phi$ and $g$ the bound on the maximum $r$ required can be significantly lowered.

\begin{algorithm}[tb]
   \caption{\bf (Local Descent Meta-Algorithm)}
   \label{algorithm}
\begin{algorithmic}
	 \INPUT{p - Degree of positive homogeneity for $(\phi,g)$}
	 \INPUT{$\{(X^1,\ldots,X^K)_r, \ Q \}$ - Initialization for variables}
	 \WHILE{Not Converged}
		\STATE{Perform local descent on variables $\{(X^1,\ldots,X^K),Q\}$ until arriving at a local minimum $\{ (\tilde{X}^1,\ldots,\tilde{X}^K),\tilde{Q} \}$}
		\IF{$\exists i_0 \in \{1,\ldots,r\}$ such that $(\tilde{X}^1_{i_0},\ldots,\tilde{X}^K_{i_0})=(0,\ldots,0)$}
			\STATE{$\{(\tilde{X}^1,\ldots,\tilde{X}^K),\tilde{Q} \}$ is a global minimum. Return.}
		\ELSE
			\IF {$\exists \theta \in \Re^r \backslash 0$ such that $\sum_{i=1}^r \theta_i \phi(X^1_i,\ldots,X^K_i) = 0$}
				\STATE{Scale $\theta$ so that $\min_i \theta_i = -1$}
				\STATE{Set $(X^1_i,\ldots,X^K_i) = ((1-\theta_i)^{1/p} \tilde{X}^1_i,\ldots,(1-\theta_i)^{1/p}\tilde{X}^K_i), \ \forall i \in \{1, \ldots, r\}$}
			\ELSE
				\STATE{Increase size of factorized variables by appending an all zero slice}
				\STATE{$(X^1,\ldots,X^K)_{r+1} = ([\tilde{X}^1 \ 0],\ldots,[\tilde{X}^K \ 0])$}
			\ENDIF
			\STATE{Set $Q=\tilde{Q}$}
			\STATE{Continue loop}
		\ENDIF
	 \ENDWHILE
\end{algorithmic}
\end{algorithm}

\begin{corollary}
	Algorithm \ref{algorithm} will find a global minimum of $f_r(X^1,\ldots,X^K,Q)$ as defined in \eqref{eq:factor_obj}.  If $r$ is intialized to be greater than $\card(D)$, then the size of the factorized variables will not increase.  Otherwise, the algorithm will terminate with $r \leq \card(D)+1$.
\end{corollary}

%% file: PosFactor_conclusions.tex
\section{Discussion and Conclusions}

We begin the discussion of our results with a cautionary note; namely, these results can be challenging to apply in practice.  In particular, many algorithms based on alternating minimization can typically only guarantee convergence to a critical point, and with the inherent non-convexity of the problem, verifying whether a given critical point is also a local minima can be a challenging problem on its own.  Nevertheless, we emphasize that our results guarantee that global minimizers can be found from purely local descent if the optimization problem falls within the general framework we have described here.  As a result, even if the particular local descent strategy one chooses for a specific problem does not come with guaranteed convergence to a local minimum, the scope of the problem is still vastly reduced from a full global optimization.  There is no need, in theory, to consider multiple initializations or more complicated (and much larger scale) techniques to explore the entire search space. 

\subsection{Balanced Degrees of Homogeneity}

In addition to the above points, our analysis analysis also provides a few insights into the behavior of factorization problems and offers simple guidance on the design of such problems.  The first is that balancing the degree of positive homogeneity between the regularization function and the mapping function is crucial.  Here we have analyzed a mapping $\Phi$ with the particular form given in \eqref{eq:Phi_r_def}.  We conjecture our results can likely be generalized to include additional factorization mappings (which we save for future work), but even for more general mappings and regularization functions, requiring the degrees of positive homogeneity to match between the regularization function and the mapping function will be critical to showing results similar to those we present here.  In general, if the degrees of positive homogeneity do not match between the factorization mapping and the regularization function, then it either becomes impossible to make guarantees regarding the global optimality of a local minimum, or the regularization function does nothing to limit the size of the factorization, so the degrees of freedom in the model are largely determined by the user defined choice of $r$.  As a demonstration of these phenomena, first consider the case where we have a general mapping $\Phi(X^1,\ldots,X^K)$ which is positively homogeneous with degree $p$ (but which is not assumed to have form \eqref{eq:Phi_r_def}).  Now, consider a general regularization function $G(X^1,\ldots,X^K)$ which is positively homogeneous with degree $p' < p$, then the following proposition provides a simple counter-example demonstrating that in general it is not possible to guarantee that a global minimum can be found from local descent.

\begin{proposition}
\label{prop:pos_degree_match}
Let $\ell : \Re^D \rightarrow \Re$ be a convex function with $\partial \ell(0) \neq \emptyset$; let $\Phi : \Re^{D^1} \times \ldots \times \Re^{D^K} \rightarrow \Re^D$ be a positively homogeneous mapping with degree $p$; and let $G : \Re^{D^1} \times \ldots \times \Re^{D^K} \rightarrow \Re_+$ be a positively homogeneous function with degree $p'<p$ such that $G(0,\ldots,0)=0$ and $G(X^1,\ldots,X^K)>0 \ \ \forall \{(X^1,\ldots,X^K) : \Phi(X^1,\ldots,X^K) \neq 0 \}$.  Then, the optimization problem given by
\begin{equation}
\min_{(X^1,\ldots,X^K)} \tilde{f}(X^1,\ldots,X^K) = \ell(\Phi(X^1,\ldots,X^K)) + G(X^1,\ldots,X^K)
\end{equation}
will always have a local minimum at $(X^1,\ldots,X^K) = (0,\ldots,0)$.  Additionally, $\forall (X^1,\ldots,X^K)$ such that $\Phi(X^1,\ldots,X^K) \neq 0$ there exists a neighborhood such that $\tilde{f}(\epsilon X^1,\ldots,\epsilon X^K) > \tilde{f}(0,\ldots,0)$ for $\epsilon > 0$ and sufficiently small.
\end{proposition}
\begin{proof}
Consider $\tilde{f}(\epsilon X^1,\ldots,\epsilon X^K) - \tilde{f}(0,\ldots,0)$.  This gives
\begin{align}
&\ell(\Phi(\epsilon X^1,\ldots,\epsilon X^K)) + G(\epsilon X^1,\ldots,\epsilon X^K) - \ell(0) - G(0,\ldots,0) = \\
&\ell(\epsilon^p \Phi(X^1,\ldots,X^K)) -\ell(0) + \epsilon^{p'} G(X^1,\ldots,X^K) \geq \\
&\epsilon^p \left< \partial \ell(0),\Phi(X^1,\ldots,X^K) \right> + \epsilon^{p'} G(X^1,\ldots,X^K)
\end{align}
Recall that $p>p'$ and $\Phi(X^1,\ldots,X^K) \neq 0 \implies G(X^1,\ldots,X^K) > 0$, so $\forall(X^1,\ldots,X^K)$, $\tilde{f}(\epsilon X^1,\ldots,\epsilon X^K)-\tilde{f}(0,\ldots,0) \geq 0$ for $\epsilon > 0$ and sufficiently small, with equality iff $G(X^1,\ldots,X^K) = 0  \implies \Phi(X^1,\ldots,X^K)=0$, giving the result. 
\end{proof}

The above proposition shows that unless we have the special case where $(X^1,\ldots,X^K) = (0,\ldots,0)$ happens to be a global minimizer, then there will always exist a local minimum at the origin, and from the origin it will always be necessary to take an increasing path to escape the local minimum.  The case described above, where $p > p'$, is arguably the more common situation for mismatched degrees of homogeneity (as opposed to $p < p'$), and a typical example might be an objective function such as
\begin{equation}
	\ell(\Phi(X^1,\ldots,X^K)) + \lambda \sum_{i=1}^K \|X^i\|^{p'}
\end{equation}
where $\Phi$ is a positively homogeneous mapping with degree $K>2$ (e.g., the mapping of a deep neural network) but $p'$ is typcially taken to be only 1 or 2 depending on the particular choice of norm.

Conversely, in the situation where $p' > p$, then it is often the case that the regularization function is not sufficient to 'limit' the size of the factorization, in the sense that the objective function can always be decreased by allowing the size of the factors to grow.  As a simple example, consider the case of matrix factorization with the objective function
\begin{equation}
	\ell(UV^T) + \lambda (\|U\|^{p'} + \|V\|^{p'})
\end{equation}
If the size of the factorization doubles, then we can always take $[\tfrac{\sqrt{2}}{2} U \ \tfrac{\sqrt{2}}{2}U][\tfrac{\sqrt{2}}{2} V \ \tfrac{\sqrt{2}}{2} V]^T = UV^T$, so if $(\tfrac{\sqrt{2}}{2})^{p'} ( \|[U \ U]\|^{p'} + \|[V \ V]\|^{p'} ) < \|U\|^{p'} + \|V\|^{p'}$, then the objective function can always be decreased by simply duplicating and scaling the existing factorization.  It is easily verified that the above inequality is satisfied for many choices of norm (for example, all the $l_q$ norms with $q \geq 1$) when $p'>2$.  As a result, this implies that the degrees of freedom in the model will be largely dependent on the initial choice of the number of columns in $(U,V)$, since in general the objective function is typically decreased by having all entries of $(U,V)$ be non-zero.

\subsection{Implications for Neural Networks}

Examining our results specifically as they apply to deep neural networks, we first note that from our analysis we have shown that neural networks which are based on positively homogeneous mappings can be regularized in the way we have outlined in our framework so that the optimization problem of training the network can be analyzed from a convex framework.  Further, we suggest that our results provide a partial explanation to the recently observed empirical phenomenon where replacing the traditional sigmoid or hyperbolic tangent non-linearities with positively homogeneous non-linearities, such as rectification and max-pooling, significantly boosts the speed of optimization and the performance of the network \citep{Dahl:ICASSP2013,Maas:ICML2013, Krizhevsky:NIPS2012,Zeiler:ICASSP2013}.  Namely, by using a positively homogeneous network mapping, the problem then becomes a convex function of the network outputs.  Additionally, we have also shown that if the size of the network is allowed to be large enough then for any initialization a global minimizer can be found from purely local descent, and thus local minima are all equivalent.  This is a similar conclusion to the work of \citet{Choromanska:2014}, who analyzed the problem from a statistical standpoint and showed that with sufficiently large networks and appropriate assumptions about the distributions of the data and network weights, then with high probability any family of networks learned with different initializations will have similar objective values, but we note that our results allow for a well defined set of conditions which will be sufficient to guarantee the property.  Finally, many modern large scale networks do not use traditional regularization on the network weigh parameters such as an $l_1$ or $l_2$ norms during training and instead rely on alternative forms of regularization such as dropout as it tends to achieve better performance in practice\citep{Srivastava:JMRL2014,Krizhevsky:NIPS2012,Wan:ICML2013}.  Given our commentary above regarding the critical importance of balancing the degree of homogeneity between the mapping and the regularizer, an immediate prediction of our analysis is that simply ensuring that the degrees of homogeneity are balanced could be a significant factor in improving the performance of deep networks.

We conclude by noting that the main limitation of our current framework in the context of the analysis of currently existing state-of-the-art neural networks is that the form of the mapping we study here \eqref{eq:Phi_r_def} implies that the network architecture must consist of $r$ parallel subnetworks, where each subnetwork has a particular architecture defined by the elemental mapping $\phi$.  Previously, we mentioned as an example that the well known ImageNet network from \citep{Krizhevsky:NIPS2012} can be described by our framework by taking $r=1$ and using an appropriate definition of $\phi$; however, to apply Corollary \ref{corr:r1} to then test for global optimality, we must test whether it is possible to reduce the objective function by adding an entire network with the same architecture in parallel to the given network.  Clearly, this is a significant limitation for the application of these results and suggests two possibilities for future work.  The first is that simply implementing neural networks with a highly parallel network architecture and relatively simple subnetwork architectures could be advantageous and worthy of experimental study.  In fact, the ImageNet network already has a certain degree of parallelization as the initial convolutional layers of the network operate largely in parallel on separate GPU units.  More generally, here we have focused on mappings with form \eqref{eq:Phi_r_def} as it is conducive to analysis, but we believe that many of the results we have presented here can be generalized to more general mappings (and thus more general network architectures) using many of the principles and analysis techniques we have presented here; an extension we reserve for future work.

\subsection{Conclusions}

Here we have presented a general framework which allows for a wide variety of non-convex factorization problems to be analyzed with tools from convex analysis.  In particular, we have shown that for problems which can be placed in our framework, any local minimum can be guaranteed to be a global minimum of the non-convex factorization problem if one slice of the factorized tensors is all zero.  Additionally, we have shown that if the non-convex factorization problem is done with factors of sufficient size, then from any feasible initialization it is always possible to find a global minimizer using a purely local descent algorithm.

%% file: PosFactor.bbl
\begin{thebibliography}{29}
\providecommand{\natexlab}[1]{#1}
\providecommand{\url}[1]{\texttt{#1}}
\expandafter\ifx\csname urlstyle\endcsname\relax
  \providecommand{\doi}[1]{doi: #1}\else
  \providecommand{\doi}{doi: \begingroup \urlstyle{rm}\Url}\fi

\bibitem[Aharon et~al.(2006)Aharon, Elad, and Bruckstein]{Elad:TSP06}
M.~Aharon, M.~Elad, and A.~M. Bruckstein.
\newblock {K-SVD:} an algorithm for designing overcomplete dictionaries for
  sparse representation.
\newblock \emph{IEEE Trans. on Signal Processing}, 54\penalty0 (11):\penalty0
  4311--4322, 2006.

\bibitem[Bach(2013)]{bach13}
F.~Bach.
\newblock Convex relaxations of structured matrix factorizations.
\newblock \emph{arXiv:1309.3117v1}, 2013.

\bibitem[Bach et~al.(2008)Bach, Mairal, and Ponce]{bach08}
F.~Bach, J.~Mairal, and J.~Ponce.
\newblock Convex sparse matrix factorizations.
\newblock \emph{arXiv:0812.1869v1}, 2008.

\bibitem[Bach(2014)]{Bach:2014}
Francis Bach.
\newblock Breaking the curse of dimensionality with convex neural networks.
\newblock \emph{arXiv preprint arXiv:1412.8690}, 2014.

\bibitem[Bengio et~al.(2005)Bengio, Roux, Vincent, Delalleau, and
  Marcotte]{Bengio:NIPS2005}
Yoshua Bengio, Nicolas~L Roux, Pascal Vincent, Olivier Delalleau, and Patrice
  Marcotte.
\newblock Convex neural networks.
\newblock In \emph{Advances in neural information processing systems}, pages
  123--130, 2005.

\bibitem[Burer and Monteiro(2005)]{burer05}
S.~Burer and R.~D.~C. Monteiro.
\newblock Local minima and convergence in low-rank semidefinite programming.
\newblock \emph{Mathematical Programming, Series A}, \penalty0 (103):\penalty0
  427---444, 2005.

\bibitem[Cai et~al.(2008)Cai, Cand\'es, and Shen]{Cai:SJO08}
J-F Cai, E.~J. Cand\'es, and Z.~Shen.
\newblock A singular value thresholding algorithm for matrix completion.
\newblock \emph{SIAM Journal of Optimization}, 20\penalty0 (4):\penalty0
  1956--1982, 2008.

\bibitem[Choromanska et~al.(2014)Choromanska, Henaff, Mathieu, Arous, and
  LeCun]{Choromanska:2014}
Anna Choromanska, Mikael Henaff, Michael Mathieu, G{\'e}rard~Ben Arous, and
  Yann LeCun.
\newblock The loss surface of multilayer networks.
\newblock \emph{arXiv preprint arXiv:1412.0233}, 2014.

\bibitem[Cichocki et~al.(2009)Cichocki, Zdunek, Phan, and Amari]{Cichocki:2009}
Andrzej Cichocki, Rafal Zdunek, Anh~Huy Phan, and Shun-ichi Amari.
\newblock \emph{Nonnegative matrix and tensor factorizations: applications to
  exploratory multi-way data analysis and blind source separation}.
\newblock John Wiley \& Sons, 2009.

\bibitem[Dahl et~al.(2013)Dahl, Sainath, and Hinton]{Dahl:ICASSP2013}
George~E Dahl, Tara~N Sainath, and Geoffrey~E Hinton.
\newblock Improving deep neural networks for lvcsr using rectified linear units
  and dropout.
\newblock In \emph{Acoustics, Speech and Signal Processing (ICASSP), 2013 IEEE
  International Conference on}, pages 8609--8613. IEEE, 2013.

\bibitem[Gandy et~al.(2011)Gandy, Recht, and Yamada]{Gandy:InvProb2011}
Silvia Gandy, Benjamin Recht, and Isao Yamada.
\newblock Tensor completion and low-n-rank tensor recovery via convex
  optimization.
\newblock \emph{Inverse Problems}, 27\penalty0 (2):\penalty0 025010, 2011.

\bibitem[Haeffele et~al.(2014)Haeffele, Young, and Vidal]{Haeffele:ICML14}
B.~Haeffele, E.~Young, and R.~Vidal.
\newblock Structured low-rank matrix factorization: Optimality, algorithm, and
  applications to image processing.
\newblock In \emph{International Conference on Machine Learning}, 2014.

\bibitem[Kolda and Bader(2009)]{Kolda:SIAMrev2009}
Tamara~G Kolda and Brett~W Bader.
\newblock Tensor decompositions and applications.
\newblock \emph{SIAM review}, 51\penalty0 (3):\penalty0 455--500, 2009.

\bibitem[Krizhevsky et~al.(2012)Krizhevsky, Sutskever, and
  Hinton]{Krizhevsky:NIPS2012}
Alex Krizhevsky, Ilya Sutskever, and Geoffrey~E Hinton.
\newblock Imagenet classification with deep convolutional neural networks.
\newblock In \emph{Advances in neural information processing systems}, pages
  1097--1105, 2012.

\bibitem[Lee and Seung(1999)]{Lee:Nature1999}
Daniel~D Lee and H~Sebastian Seung.
\newblock Learning the parts of objects by non-negative matrix factorization.
\newblock \emph{Nature}, 401\penalty0 (6755):\penalty0 788--791, 1999.

\bibitem[Maas et~al.(2013)Maas, Hannun, and Ng]{Maas:ICML2013}
Andrew~L Maas, Awni~Y Hannun, and Andrew~Y Ng.
\newblock Rectifier nonlinearities improve neural network acoustic models.
\newblock In \emph{Proc. ICML}, volume~30, 2013.

\bibitem[Mairal et~al.(2010)Mairal, Bach, Ponce, and Sapiro]{Mairal:JMLR2010}
Julien Mairal, Francis Bach, Jean Ponce, and Guillermo Sapiro.
\newblock Online learning for matrix factorization and sparse coding.
\newblock \emph{The Journal of Machine Learning Research}, 11:\penalty0 19--60,
  2010.

\bibitem[Ngiam et~al.(2011)Ngiam, Coates, Lahiri, Prochnow, Le, and
  Ng]{Ngiam:ICML2011}
Jiquan Ngiam, Adam Coates, Ahbik Lahiri, Bobby Prochnow, Quoc~V Le, and
  Andrew~Y Ng.
\newblock On optimization methods for deep learning.
\newblock In \emph{Proceedings of the 28th International Conference on Machine
  Learning (ICML-11)}, pages 265--272, 2011.

\bibitem[Recht et~al.(2010)Recht, Fazel, and Parrilo]{Recht:SIAM10}
B.~Recht, M.~Fazel, and P.~Parrilo.
\newblock Guaranteed minimum-rank solutions of linear matrix equations via
  nuclear norm minimization.
\newblock \emph{SIAM Review}, 52\penalty0 (3):\penalty0 471--501, 2010.

\bibitem[Rockafellar and Wets(1998)]{Rockafellar:1998}
R~Tyrrell Rockafellar and Roger J-B Wets.
\newblock \emph{Variational Analysis}, volume 317.
\newblock Springer, 1998.

\bibitem[Rumelhart et~al.(1988)Rumelhart, Hinton, and
  Williams]{Rumelhart:CogModel1988}
David~E Rumelhart, Geoffrey~E Hinton, and Ronald~J Williams.
\newblock Learning representations by back-propagating errors.
\newblock \emph{Cognitive modeling}, 5, 1988.

\bibitem[Srivastava et~al.(2014)Srivastava, Hinton, Krizhevsky, Sutskever, and
  Salakhutdinov]{Srivastava:JMRL2014}
Nitish Srivastava, Geoffrey Hinton, Alex Krizhevsky, Ilya Sutskever, and Ruslan
  Salakhutdinov.
\newblock Dropout: A simple way to prevent neural networks from overfitting.
\newblock \emph{The Journal of Machine Learning Research}, 15\penalty0
  (1):\penalty0 1929--1958, 2014.

\bibitem[Tomioka et~al.(2010)Tomioka, Hayashi, and Kashima]{Tomioka:2010}
Ryota Tomioka, Kohei Hayashi, and Hisashi Kashima.
\newblock Estimation of low-rank tensors via convex optimization.
\newblock \emph{arXiv preprint arXiv:1010.0789}, 2010.

\bibitem[Wan et~al.(2013)Wan, Zeiler, Zhang, Cun, and Fergus]{Wan:ICML2013}
Li~Wan, Matthew Zeiler, Sixin Zhang, Yann~L Cun, and Rob Fergus.
\newblock Regularization of neural networks using dropconnect.
\newblock In \emph{Proceedings of the 30th International Conference on Machine
  Learning (ICML-13)}, pages 1058--1066, 2013.

\bibitem[Wright and Nocedal(1999)]{Wright:1999}
Stephen~J Wright and Jorge Nocedal.
\newblock \emph{Numerical optimization}, volume~2.
\newblock Springer New York, 1999.

\bibitem[Xu and Yin(2013)]{Xu:SIAM2013}
Yangyang Xu and Wotao Yin.
\newblock A block coordinate descent method for regularized multiconvex
  optimization with applications to nonnegative tensor factorization and
  completion.
\newblock \emph{SIAM Journal on imaging sciences}, 6\penalty0 (3):\penalty0
  1758--1789, 2013.

\bibitem[Yu et~al.(2014)Yu, Zhang, and Schuurmans]{Yu:2014}
Yaoliang Yu, Xinhua Zhang, and Dale Schuurmans.
\newblock Generalized conditional gradient for sparse estimation.
\newblock \emph{arXiv preprint arXiv:1410.4828}, 2014.

\bibitem[Zeiler et~al.(2013)Zeiler, Ranzato, Monga, Mao, Yang, Le, Nguyen,
  Senior, Vanhoucke, Dean, et~al.]{Zeiler:ICASSP2013}
Matthew~D Zeiler, M~Ranzato, Rajat Monga, M~Mao, K~Yang, Quoc~Viet Le, Patrick
  Nguyen, A~Senior, Vincent Vanhoucke, Jeffrey Dean, et~al.
\newblock On rectified linear units for speech processing.
\newblock In \emph{Acoustics, Speech and Signal Processing (ICASSP), 2013 IEEE
  International Conference on}, pages 3517--3521. IEEE, 2013.

\bibitem[Zhang et~al.(2013)Zhang, Yu, and Schuurmans]{Zhang:NIPS2013}
Xinhua Zhang, Yao-Liang Yu, and Dale Schuurmans.
\newblock Polar operators for structured sparse estimation.
\newblock In \emph{Advances in Neural Information Processing Systems}, pages
  82--90, 2013.

\end{thebibliography}
